%% file: v3.tex
\documentclass[11pt,a4paper,oneside]{amsart}
\input{top}

\newcounter{commentcounter}

\newtheorem*{proposition*}{Proposition}
\newtheorem*{definition*}{Definition}
\newtheorem*{notation*}{Notation}
\newtheorem*{lemma*}{Lemma}
\newtheorem*{claim*}{Claim}

\newcommand{\mc}[1]{\mathcal{#1}}

\usepackage{physics}

\def\N{{\mathbb N}}

\def\T{{\mathcal T}}
\newcommand{\length}{\textrm{length}}
\newcommand{\cone}{\textrm{Cone}_\omega}
\newcommand{\QREG}{\QQ\mathbf{REG}}

\DeclarePairedDelimiter\absval{\lvert}{\rvert}

\title{Regularity of quasigeodesics characterises hyperbolicity}
\usepackage[foot]{amsaddr}
\author{Sam Hughes}
\email{sam.hughes.maths@gmail.com}\email{hughes@math.uni-bonn.de}
\author{Patrick S. Nairne}
\email{nairne@maths.ox.ac.uk}
\author{Davide Spriano}
\email{spriano@maths.ox.ac.uk}
\address{Mathematical Institute, Andrew Wiles Building, Observatory Quarter, University of Oxford, Oxford OX2 6GG, UK}
\date{\today}

\begin{document}

\begin{abstract}
We characterise hyperbolic groups in terms of quasigeodesics in the Cayley graph forming regular languages.  We also obtain a quantitative characterisation of hyperbolicity of geodesic metric spaces by the non-existence of certain local $(3,0)$-quasigeodesic loops.  As an application we make progress towards a question of Shapiro regarding groups admitting a uniquely geodesic Cayley graph.
\end{abstract}

\maketitle

\section{Introduction}
Hyperbolic groups were introduced by Gromov \cite{Gromov1987} and revolutionised the study of finitely generated groups. Arguably, their most remarkable feature is that hyperbolicity connects several, and at a first glance independent, areas of mathematics. Confirming this, there are several different characterisations of hyperbolicity --- such as the geometric thin triangle condition \cite{Gromov1987}, the dynamical characterisation via convergence actions \cite{Bowditch1998}, surjectivity of the comparison map in bounded cohomology \cite{Mineyev2001,Mineyev2002,Franceschini2018} and vanishing of $\ell^\infty$-cohomology \cite{Gersten1998}, linear isoperimetric inequality \cite{Gromov1987}, all asymptotic cones being $\RR$-trees \cite{Gromov1987}, and others \cite{Gromov1987,Gromov1993,Papasoglu1995b,AllcockGersten1999, Gilman02, ChatterjiNiblo2007,Wenger2008}.

Another significant feature of hyperbolic groups is that they present very strong algorithmic properties. Most notably, they have solvable isomorphism problem \cite{DahmaniGuirardelIsomorphism, SelaIsomorphism}, they are biautomatic \cite{EpsteinEtAl1992} and so the word problem can be solved via finite state automata, and sets of their rational quasigeodesics form a regular language \cite{HoltRees2003}.

This last property will be a central focus in the paper, and we call it rational regularity, or for short $\QREG$.

\begin{defn}\label{defn.QREG}
A finitely generated group $G$ is $\QREG$ if for all rational $\lambda \geq 1$, real $\epsilon \geq 0$ and finite generating sets $S$, the $(\lambda, \epsilon)$--quasigeodesics in the Cayley graph $\Gamma(G,S)$ form a regular language. 
\end{defn}

As mentioned, in \cite{HoltRees2003} Holt and Rees prove that every word hyperbolic group is $\QREG$. It is natural to ask if this provides a characterization of hyperbolic groups, as was conjectured in \cite[Problem~1]{CordesRussellSprianoZalloum2020}. The main result of the paper is the following.

\begin{thm}\label{thm.main}
A finitely generated group is hyperbolic if and only if it is $\QREG$.
\end{thm}
We remark that it is necessary to not consider only geodesics. In \cite{Cannon1984}, Cannon proved that for any finite generating set the geodesics in a hyperbolic group form a regular language. However, this does not characterise hyperbolicity: Neumann and Shapiro \cite[Propositions 4.1; 4.4]{NeumannShapiro1995} prove that for any finite generating set the geodesics in an abelian group form a regular language.

In fact, we prove the following strong converse to the result of Holt and Rees.
\begin{thm}\label{thm.converse}
Let $G$ be a finitely generated non-hyperbolic group. Then for all finite generating sets $S$ and all $\lambda' > 54$ the set of $(\lambda',0)$-quasigeodesics is not regular in $\Gamma(G,S)$. 
\end{thm}
Once this is proved, it is clear that non-hyperbolic groups can never be $\QREG$. 

To prove Theorem~\ref{thm.converse} we need a strong quantitative characterisation of hyperbolicity. It is known that a geodesic metric space is hyperbolic if and only if local quasigeodesics are global quasigeodesics \cite[Proposition~7.2.E]{Gromov1987}. More precisely, the contrapositive can be stated as follows: a space is non-hyperbolic if and only if there exists a pair of constants $(\lambda, \epsilon)$, a sequence $L_n \rightarrow \infty$ and a sequence of $L_n$-locally $(\lambda,\epsilon)$-quasigeodesic paths which are not global $(\lambda',\epsilon')$-quasigeodesics for any uniform choice of constants $(\lambda',\epsilon')$. Hence, \emph{a priori}, to check for hyperbolicity one would want to consider all choices of $(\lambda, \epsilon)$ and all choices of locally $(\lambda,\epsilon)$-quasigeodesic paths. We get around this using a criterion of Hume and Mackay (\cite{HumeMackay:poorly}) that essentially states that one needs only consider $L$-locally $(18,0)$-quasigeodesic loops whose length is comparable to $L$. To prove Theorem~\ref{thm.converse}, we use such a sequence of loops to contradict the pumping lemma. 

One of the limitations of the criterion of Hume and Mackay is that it only works for graphs, as it relies on Papazoglu's bigon criterion \cite{Papasoglu1995b}, which is false for general geodesic metric spaces.  We remark that a version of the bigon criterion for general metric spaces appeared in the master's thesis of Pomroy \cite{Pomroy}, a proof can be found in \cite[Appendix]{ChatterjiNiblo2007}. We develop a characterization of hyperbolicity, analogous to Hume and Mackay, that works for all geodesic metric spaces and does not rely on any bigon criteria. 

\begin{thm}\label{thm:main.technical} A geodesic metric space $X$ is not hyperbolic if and only if there exists a sequence $L_n \rightarrow \infty$ and a sequence of non-constant $L_n$-locally $(3,0)$-quasigeodesic loops $\gamma_n$ that satisfy $\ell(\gamma_n) \leq K L_n$, where $K$ is some constant that does not depend on $n$. \end{thm}

Note that requiring that the loops are non-constant implies that $\diam (\gamma_n)$ diverges to infinity. Indeed,  $(3,0)$-quasigeodesics are injective and for $\gamma_n$ to be a loop we must have that $[0,L_n]$ is a proper subset of the domain. So, morally Theorem~\ref{thm:main.technical} states that in a non-hyperbolic group there are paths that are very well-behaved locally, but globally are loops of increasing size, hence cannot be global quasi-geodesics for any choice of constants.

Although striking, the presence of a sharp gap in the behaviour of local-quasigeodesics is not surprising. For instance, it is known that the Dehn function of a finitely presented group has a gap. A deep theorem of S. Wenger 
\cite{WengerIsoperim}, extending results of \cite{Gromov1987,Olshanski1991,Bowditch1995,Papasoglu1995}, shows that if the isoperimetric function satisfies $D(x) \leq \frac{1-\epsilon}{4\pi}x^2$, then it is in fact linear. 

Our strategy in proving Theorem~\ref{thm:main.technical} relies on the study of asymptotic cones of metric spaces. If $X$ is non-hyperbolic, then there is an asymptotic cone that is not a tree \cite[2.A]{Gromov1993}, and it contains a simple loop. By using a series of approximations, we exploit this loop to produce a family of loops of controlled length that are locally $(3,0)$--quasigeodesic.

\subsection*{A question of Shapiro}
A natural class of graphs to consider is the class of geodetic graphs. A connected graph is called \emph{geodetic} if for any pair of vertices there is exactly one geodesic connecting them. In \cite{Shapiro1997}, M. Shapiro asked when a group admits a (locally finite) geodetic Cayley graph. He conjectures that such a group needs to be \emph{plain}, that is, a free product where the factors are either free or finite. 
In \cite{Papasoglu1993}, Papasoglu proved that a geodetic hyperbolic group is virtually free. It is still open whether all geodetic groups are hyperbolic, and whether all geodetic virtually free groups are plain. 

We provide an answer to the first implication under an additional, language theoretic, assumption.

\begin{thm}\label{thm:UG_Hyperbolic}
Let $G$ be a finitely generated group with a generating set $S$ such that $\Gamma(G, S)$ is geodetic.  If there exists $\lambda>3$ such that the language of $(\lambda, 0)$--quasigeodesics is regular, then $G$ is hyperbolic and hence virtually free.
\end{thm}


\subsection*{Structure of the paper}
In \Cref{sec:background} we give the necessary background on Cayley graphs, hyperbolic metric spaces, quasigeodesics, languages, automata, and asymptotic cones.  In \Cref{sec:proofs} we provide proofs of the results from the introduction.  More specifically, in \Cref{sec.cones} we prove \Cref{thm:main.technical} and compare it to \cite[Proposition~1.5]{HumeMackay:poorly}, in \Cref{sec.reg} we prove \Cref{thm.main} and \Cref{thm.converse}, and in \Cref{sec.geod} we prove \Cref{thm:UG_Hyperbolic}.

\subsection*{Acknowledgements}
The first author was supported by the European Research Council (ERC) under the European Union’s Horizon 2020 research and innovation programme (Grant agreement No. 850930). The third author was partly supported by the Christ Church Research Centre. We would like to thank  the LabEx of the Institut Henri Poincar\'e (UAR 839 CNRS-Sorbonne Universit\'e) for their support during the trimester program ``Groups acting on Fractals, Hyperbolicity and Self-similarity''.  The second author would like to thank their supervisor Cornelia Dru\textcommabelow{t}u for useful discussions on these topics. We are also grateful to Panos Papasoglu  and Sarah Rees for a number of helpful conversations, to Murray Elder and Adam Piggott for helpful conversations regarding Shapiro's question, and to Elia Fioravanti for informing us of Hume--Mackay's characterization of hyperbolicity. Finally, we would like to thank the anonymous referees that carefully read the paper and suggested a number of clarifications and improvements.

\section{Background}\label{sec:background}
\subsection{Cayley graphs, hyperbolicity, and quasigeodesics}
Let $G$ be a finitely generated group with generating set $S$. We denote by $\Gamma(G,S)$ the \emph{Cayley graph} of $G$ with respect to $S$, that is, the graph with vertices $G$ and edges $\{g,gs\}$ where $g\in G$ and $s\in S$.
We denote by $\absval{g}$ the word-length of $g$ with respect to $S$; equivalently, this is equal to $d_{\Gamma(G,S)}(e,g)$.

Let $\delta\geq 1$.  A metric space $X$ is $\delta$-\emph{hyperbolic} if every geodesic triangle in $X$ is $\delta$-thin.  Here a geodesic triangle is $\delta$-\emph{thin} if every edge is contained in the $\delta$-neighbourhood of the two other edges.  We say a finitely generated group $G$ is \emph{hyperbolic} if the Cayley graph $\Gamma(G,S)$ is a $\delta$-hyperbolic metric space for some finite generating set $S$.

Let $\lambda\geq1$ and $\epsilon\geq0$.  Given metric spaces $X$ and $Y$, a $(\lambda,\epsilon)$-\emph{quasi-isometric embedding} $f\colon X\to Y$ is a function satisfying
\[\frac{1}{\lambda} d_X(x,y) - \epsilon \leq d_Y(f(x),f(y)) \leq \lambda d_X(x,y) + \epsilon\]
for all $x,y \in X$.

\begin{defn}
In the context of a Cayley graph $\Gamma(G,S)$, by a \textit{$(\lambda,\epsilon)$-quasigeodesic} of length $a \in \N$ we mean a simplicial path $c: [0,a] \rightarrow \Gamma(G,S)$ such that for any two integers $x,y \in [0,a]$ we have
\[d_{[0,a]}(x,y) \leq \lambda d_G(c(x),c(y)) + \epsilon\]
Equivalently, a $(\lambda,\epsilon)$-quasigeodesic is a finite word $w$ on the alphabet $S \cup S^{-1}$ for which $\length(u) \leq \lambda d_G(e,u) + \epsilon$ for all subwords of $u$ of $w$. 
\end{defn}

\begin{defn} \label{def.qgmetricspace}
In the context of an arbitrary metric space $X$, a $(\lambda,\epsilon)$-\emph{quasigeodesic} of length $a>0$ in $X$ is a quasiisometric embedding $c\colon [0,a]\to X$.
\end{defn}

Given a path $c\colon [0,a]\to \Gamma(G,S)$ or $c\colon [0,a]\to X$, we say that $c$ is an $L$-\emph{locally} $(\lambda,\epsilon)$-\emph{quasigeodesic} or an \emph{$(L,\lambda,\epsilon)$-local-quasigeodesic} if $c$ restricted to each subset of $[0,a]$ of length $L$ is a $(\lambda,\epsilon)$--quasigeodesic. We say $c$ is a $(L, \lambda,\epsilon)$-\emph{quasigeodesic loop} if, in addition, $c(0)=c(a)$. We denote the length of a path $c$ by $\ell(c)$.

\subsection{Regular languages and automata}
The following definitions are standard and may be found in \cite[Chapter~1]{EpsteinEtAl1992}. Given a finite set $A$, let $A^\star$ be the free monoid generated by $A$, i.e. the set of finite words that can be written with letters in $A$. A \emph{language} over the alphabet $A$ is a subset $L\subseteq A^\star$. A \emph{finite state automaton} (FSA) $\mc{M}$ over the alphabet $A$ consists a finite oriented graph $\Gamma(\mc{M})$, together with an edge label function $\ell\colon E(\Gamma(\mc{M})) \to A$, a chosen vertex $q_I \in V(\Gamma(\mc{M}))$ called the \emph{initial state} and subset $Q_F \subset V(\Gamma(\mc{M}))$ of \emph{final states}. The vertices of $\Gamma(\mc{M})$ are often referred to as \emph{states}. 

Let $\calm$ be an FSA over an alphabet $A$.  We say a string $w\in L$ is \emph{accepted} by $A$ if and only if there is an oriented path \(\gamma\) in \(\Gamma(\mc{M})\) starting from \(q_I\) and ending in a vertex  \(q \in Q_F\) such that $\gamma$ is labelled by $w$.  A language $L$ is \emph{regular} if and only if there exists an FSA $\calm$ such that $L$ coincides with the strings of $A^\star$ accepted by $\calm$.

Let $G$ be a group generated by a finite set $S$.  An element $w \in (S \cup S^{-1})^\ast$ labels a path in $\Gamma(G,S)$ which starts at $e$.  We say $w$ is a \emph{geodesic/$(\lambda,\epsilon)$--quasigeodesic/$(L,\lambda,\epsilon)$-local-quasigeodesic word} if it labels a path in $\Gamma(G,S)$ with the corresponding property. 

We say that the set $L^{(\lambda,\epsilon)}$ of $(\lambda,\epsilon)$--quasigeodesic words $w$ over $S \cup S^{-1}$ form the \emph{$(\lambda,\epsilon)$--quasigeodesic language of $G$ over $S$}.

\subsection{Asymptotic cones}
In this section we will give the necessary background on asymptotic cones. These concepts and definitions will only be needed for the proof of \Cref{thm:main.technical} in \Cref{sec.cones}. The idea of an asymptotic cone first appeared in the proof of Gromov's Polynomial Growth Theorem \cite{Gromov1981}, however, it was first formalised by Wilkie and van den Dries \cite{WilkievandenDries1984}.

An \emph{ultrafilter} $\omega$ on $\NN$ is a set of nonempty subsets of $\NN$ which is closed under finite intersection, upwards-closed, and if given any subset $X\subseteq\NN$, contains either $X$ or $\NN\backslash X$.  We say $\omega$ is \emph{non-principal} if $\omega$ contains no finite sets. We may equivalently view a non-principal ultrafilter $\omega$ as a finitely additive measure on the class $2^\NN$ of subsets of $\NN$ such that each subset has measure equal to $0$ or $1$, and all finite sets have measure $0$.  If some statement $P(n)$ holds for all $n\in X$ where $X\in\omega$, then we say that $P(n)$ holds \emph{$\omega$-almost surely}.

Let $\omega$ be a non-principal ultrafilter on $\NN$ and let $X$ be a metric space.  If $(x_n)_{n\in\NN}$ is a sequence of points in $X$, then a point $x$ satisfying for every $\epsilon>0$ that $\{n\ |\ d(x_n,x)\leq\epsilon\}\in\omega$, is called an \emph{$\omega$-limit} of $x_n$ and denoted by $\lim_\omega x_n$. Given a bounded sequence $x_n \in X$, there always exists a unique ultralimit $\lim_\omega x_n$. 

Let $\omega$ be a non-principal ultrafilter on $\NN$. Let $(X_n,d_n)_{n\in\NN}$ be a sequence of metric spaces with specified base-points $p_n\in X_n$.  Say a sequence $(y_n)_{n\in\NN}$ is \emph{admissible} if the sequence $(d_{X_n}(p_n,y_n))_{n\in \NN}$ is bounded.  Given admissible sequences $x=(x_n)$ and $y=(y_n)$, the sequence $(d_{X_n}(x_n,y_n))$ is bounded and we define $\hat d_\infty(x,y)\coloneqq \lim_\omega d_n(x_n,y_n)$.  Denote the set of admissible sequences by $\calx$.  For $x,y\in\calx$ define an equivalence relation by $x\sim y$ if $\hat d_\infty(x,y)=0$.  The \emph{ultralimit} of $(X_n,p_n)$ with respect to $\omega$ is the metric space $(X_\infty,d_\infty)$, where $X_\infty=\calx/\sim$ and for $[x],[y]\in X_\infty$ we set $d_\infty([x],[y])=\hat d_\infty(x,y)$. Given an admissible sequence of elements $x_n \in X_n$ we define their ultralimit in $X_\infty$ to be $\lim_\omega x_n := [(x_n)]$. Given a sequence of subsets $A_n \subset X_n$ we can define their ultralimit in $X_\infty$ to be the set $\lim_\omega(A_n) := \{{[(x_n)]} \mid  x_n\in A_n\}$, where we only consider admissible sequences $(x_n)$.

Let $\omega$ be a non-principal ultrafilter on $\NN$ and let $(\mu_n)$ be a diverging, non-decreasing sequence.   Let $(X,d)$ be a metric space and consider the sequence of metric spaces $X_n=\left(X,\frac{1}{\mu_n}d\right)$ for $n\in\NN$ with basepoints $(p_n)$.  The $\omega$-ultralimit of the sequence $(X_n,p_n)$ is called the \emph{asymptotic cone} of $X$ with respect to $\omega$, $(\mu_n)$, and $(p_n)$ and denoted $\cone(X,(\mu_n),(p_n))$.  If the sequence of basepoints is constant, then we denote the asymptotic cone by $\cone(X,(\mu_n))$. In the case of a finitely generated group, we always assume that the basepoint is the identity.

The following is \cite[Proposition~3.29(c)]{DrutuSapir2005} which we will use in the proof of \Cref{thm:main.technical}.

\begin{prop}\label{prop: Drutu Sapir}
Consider a non-principal ultrafilter $\omega$ on $\N$ and a sequence of metric spaces $(X_n, d_n)$ with basepoints $p_n \in X_n$. Suppose there exists a simple geodesic triangle in $(X_\infty, d_\infty)$. Then there exists a (possibly different) simple geodesic triangle $\Delta$ in $X_\infty$, a constant $k \geq 2$, and a sequence of simple geodesic $k$-gons $P_n$ in $X_n$ such that $\lim_\omega(P_n) = \Delta$.
\end{prop}

\section{Proofs of the results}\label{sec:proofs}


\subsection{A characterisation of hyperbolic geodesic metric spaces} \label{sec.cones}

In this section, we prove \Cref{thm:main.technical}, this is a version of a result of Hume and Mackay \cite[Proposition~5.1]{HumeMackay:poorly} for arbitrary metric spaces.

\begin{defn} \label{def.star} Let $X$ be a metric space and consider the following condition:
\begin{equation}\tag{$\star$} \label{condition.star} 
    \begin{tabular}{p{0.8\textwidth}}
    There exists an increasing sequence of positive numbers $L_n \rightarrow \infty$ and a pair of constants $K, \lambda \geq 1$ such that for every $n$ there exists a non-constant $L_n$-locally $(\lambda,0)$--quasigeodesic loop $\gamma_n$ in $X$ with $\ell(\gamma_n) \leq KL_n$.
    \end{tabular}
\end{equation}
At times, it is convenient to specify the values of the constants $K$ and $ \lambda$. In that case we say that a metric space $X$ satisfies \eqref{condition.star} with constants $(K, \lambda)$.
\end{defn}


\begin{prop}\label{prop:star_implies_non-hyp}
If a metric space $X$ satisfies \eqref{condition.star}, then $X$ is not hyperbolic. 
\end{prop}

\begin{proof}
If $X$ is hyperbolic then it satisfies the local-to-global property for quasigeodesics: for every choice of $\lambda, \epsilon$ there exist constants $L = L(\lambda, \epsilon)$, $\lambda' = \lambda'(\lambda, \epsilon)$ and $\epsilon' = \epsilon'(\lambda, \epsilon)$ such that every $L$-locally $(\lambda, \epsilon)$--quasigeodesic is a global $(\lambda',\epsilon')$--quasigeodesic. 

Suppose $X$ satisfies the local-to-global property for quasigeodesics and $X$ satisfies \eqref{condition.star} with constants $(K, \lambda)$. Let $L = L(\lambda, 0)$ be the constant given by the local-to-global property. Choose $n \in \NN$ such that $L_n \geq L$. Then $\gamma_n$ is an $L$-locally $(\lambda, 0)$--quasigeodesic. However, $\gamma_n$ is a loop and so cannot be a $(\lambda', \epsilon')$--quasigeodesic for any choice of $\lambda',\epsilon'$. 
\end{proof}

\begin{prop}\label{thm:non-hyp_implies_star} 
If $X$ is a non-hyperbolic geodesic metric space, then there exists a constant $K \geq 1$ such that $X$ satisfies \eqref{condition.star} with constants $(K,3)$.
\end{prop}

\begin{proof}
Since $X$ is not hyperbolic, there exists an ultrafilter $\omega$ and a non-decreasing scaling sequence $\mu_n$ such that $\mathrm{Cone}_\omega(X, (\mu_n))$ is not a real tree. In particular, there exists a simple geodesic triangle $\Delta \subseteq \mathrm{Cone}_\omega(X, (\mu_n))$. Using Proposition~\ref{prop: Drutu Sapir}, up to replacing $\Delta$ with another simple geodesic triangle, we obtain that $\Delta = \lim_\omega (P_n)$, where each $P_n$ is a geodesic $k$-gon in $X$ for some $k$. Let $z_n^1, \dots z_n^k$ be the vertices of $P_n$, where the labels are taken respecting the cyclic order on $P_n$. From now on, we always consider the indices mod $k$.  Denote by $e_n^{i}$ the geodesic segment connecting $z_n^i, z_n^{i+1}$, that is the appropriate restriction of $P_n$. 

Consider the points $z_\omega^i = (z_n^i)\in \Delta$, and let $e_\omega^i = \lim_\omega(e_i)$. It is a standard argument to show that $e^i_\omega$ are geodesic segments whose endpoints are $z_\omega^i$, see for instance \cite[Lemma 10.48, Exercise 10.71]{DrutuKapovich}. Since $\lim_\omega(P_n) = \Delta$, we have $e^i_\omega \subseteq \Delta$.
Since there are only $k$ edges, for any $ \rho>1$,  $\omega$--almost surely we have \begin{equation}\label{proposition:star_length_edges}\ell(e_n^i)\leq \rho \mu_n \ell(e_\omega^i).\end{equation}  In particular, $\omega$--almost surely we have $\ell(P_n) \leq \rho \mu_n \ell(\Delta)$, that is to say that the length of the polygons $P_n$ is bounded above by a linear function of $\mu_n$. Our goal is to modify the polygons $P_n$ to obtain loops that are $(c\mu_n)$--locally $(3, 0)$--quasigeodesics for some $c > 0$, and whose lengths are comparable to those of the $P_n$. 

To this end, we restrict our attention to only some edges of $P_n$. We say that an index $1 \leq i \leq k$ is \emph{active} if $e_\omega^i \neq \{z_\omega^i\}$. Let $i_1 \leq \dots \leq i_d$ be the active indices. From now on, we will only consider edges with active indices, and thus we rename $e_\omega^{i_j}$ as $a_\omega^j$ and $e_n^{i_j}$ as $a_n^j$. Thus, $a_\omega^1, \dots, a_\omega^d$ is a subdivision of the triangle $\Delta$ into a geodesic $d$-gon. Since $\Delta$ is simple and its edges are compact, we have that there exists a $\delta > 0$ such that for all edges $a_\omega^{i}$ and points $x\in a_\omega^{i}$ we have 
\[\max \{d(x, a_\omega^{i-1}), d(x, a_\omega^{i+1})\} \geq \delta.\]

For any active edge $a_n^i$ and $x\in a_n^i$, $\omega$--almost surely we have \begin{equation}\label{proposition:star_distance_edges}\max \{d(x, a_n^{i-1}), d(x, a_n^{i+1})\} \geq \delta\rho^{-1} \mu_n.\end{equation}

Further, since the terminal vertex of $a_\omega^i$ and the initial vertex of $a_\omega^{i+1}$ coincide, 
$\omega$--almost surely we have \begin{equation}\label{anotherlabel}
d(a_n^i, a_n^{i+1}) \leq \frac{1}{2} \delta \rho^{-1} \mu_n
\end{equation}

The intuitive idea is now as follows. For infinitely many $n$ we have a collection of geodesic segments ($a_n^j$) whose length keeps increasing \eqref{proposition:star_length_edges} and such that we have some control on the distance between them \eqref{proposition:star_distance_edges},\eqref{anotherlabel}. Using this, we can connect the segments to obtain loops of controlled length which are locally quasigeodesics.

More formally, fix $n$ such that \eqref{proposition:star_length_edges}, \eqref{proposition:star_distance_edges}, \eqref{anotherlabel} are satisfied and orient the $a_n^j$ with the orientation of $P_n$ that agrees with the numbering. Let $L_n = \frac{1}{2}\delta\rho^{-1} \mu_n$. From now on, we will drop the subscript $n$ and denote, for instance $L_n = L$. Let $q^1$ be the first point of $a^1$ such that $d(q^1, a^2) \leq L$. 
By the continuity of the distance function and the choice of $q^1$ we see that $d(q^1,a^2)= L$.

Let $p^2$ be a point in $a^2$ such that $d(q^1,p^2)=L$. Therefore, we see that $d(p^2,a^3)\geq \delta\rho^{-1} \mu_n >L$. Let $q^2$ be the first point in $a^2$ after $p^2$ such that $d(q^2,a^3)\leq L$. Again, we have $d(q^2,a^3)= L$, and let $p^3\in a^3$ be a point such that $d(q^2, p^3)= L$. We iterate this procedure until we obtain a point $q^d \in a^d$ and a point $p^1\in a^1$ such that $d(q^d, p^1) = L$. 

We claim that $d(p^j, q^j)\geq L$. Indeed, since $d(p^j, a^{j-1})\leq L$, \eqref{proposition:star_distance_edges} implies $d(p^j, a^{j+1})\geq \delta\rho^{-1} \mu_n = 2L$, and the result follows from the triangle inequality.

From now on, we denote by $[p^j, q^j]$ the restriction of $a^j$ between $p^j, q^j$, and we choose once and for all geodesic segments $[q^j, p^{j+1}]$ connecting $q^j, p^{j+1}$. Let $\gamma_n = \gamma$ be the concatenation 
\[\gamma = [p^1, q^1] \ast [q^1, p^2] \ast \cdots \ast [q^d, p^1],\]

where we consider $\gamma$ to be parameterized by arc length.  We will show that \(\gamma\) is a \((L; 3, 0)\)--local quasigeodesic. 

Let \(x,y\) be two points of \(\gamma\) of parameterized distance less than \(L\). We denote by $d_\gamma(x,y)$ the parameter distance. We will prove that $d_\gamma(x,y)\leq 3 d(x,y)$. If $a$ and $b$ are contained in the same segment of \(\gamma\), then the inequality clearly holds. Thus, we can assume that \(x\) and \(y\) are on two consecutive segments of $\gamma$ since the length of each segment of $\gamma$ is at least $L$.

Firstly, consider the case $x\in [p^j, q^j]$, $y\in [q^j, p^{j+1}]$. If $x = q^j$, then we would be in the previous case, so $x\neq q^j$.  We claim $d(x,y)> d(q^j,y)$. If not, this would contradict the choice of $q^j$ as the first point at distance $L$ from $a^{j+1}$. Indeed,  $d(x,y)\leq d(q^j,y)$, implies $d(x,p^{j+1})\leq d(q^j,p^{j+1})$. Therefore, $d(x,y)> d(q^j,y)$. In particular: \[d_\gamma(x,y)=d(x,q^j)+d(q^j,y)\leq \bigl(d(x,y)+d(y,q^j)\bigr)+d(q^j,y)\leq 3 d(x,y).\]

Consider now the case $x\in [q^{j-1}, p^j]$ and $y\in [p^j, q^j]$. Since $d(q^{j-1},a^j)=d(q^{j-1},p^j)$, we have $d(x,y)\geq d(x,p^j)$. Hence \[d_\gamma(x,y)=d(x,p^j)+d(p^j,y)\leq d(x,p^j)+\bigl(d(p^j,x)+d(x,y)\bigr)\leq 3 d(x,y).\]

Thus, \(\gamma\) is a \((L;3,0)\)--local quasigeodesic, where $L= L_n = \frac{1}{2}\delta\rho^{-1} \mu_n$. To conclude the proposition, we need to bound the length of $\gamma$ linearly in terms of $\mu_n$. However, observe that $d(q^j, p^{j+1}) = L$ for all $j$, and $d(p^j, q^j) \leq \ell(a^j_n) \leq \rho \mu_n \ell (a^j_\infty)$. Setting $M= \max \ell(a^j_\infty)$, we obtain 
\[\ell(\gamma)\leq d\left(\frac{1}{2}\delta\rho^{-1} \mu_n + \rho M \mu_n\right) = d\left(\frac{1}{2}\delta\rho^{-1}  + \rho M\right) \mu_n.\qedhere\]
\end{proof}

\begin{duplicate}[\Cref{thm:main.technical}]
    A geodesic metric space $X$ is not hyperbolic if and only if there exists a sequence $L_n \rightarrow \infty$ and a sequence of non-constant $L_n$-locally $(3,0)$-quasigeodesic loops $\gamma_n$ that satisfy $\ell(\gamma_n) \leq K L_n$, where $K$ is some constant that does not depend on $n$. 
\end{duplicate}
\begin{proof}
One direction is given by \Cref{thm:non-hyp_implies_star} and the other by \Cref{prop:star_implies_non-hyp}.
\end{proof}

In \cite{HumeMackay:poorly}, by an \textit{$18$-bilipschitz embedded cyclic subgraph in $X$}, Hume and Mackay mean an injective graph homomorphism $\phi: C_n \rightarrow X$ of the circular graph $C_n$ into the graph $X$ such that $d_{C_n}(x,y) \leq 18 d_{X}(\phi(x),\phi(y))$ for all vertices $x,y \in C_n$. 
We now compare our results above to the following proposition of Hume and Mackay. 

\begin{prop}~\cite[Proposition~5.1]{HumeMackay:poorly}\label{prop.hm}
Let $X$ be a connected graph. $X$ is hyperbolic if and only if there is some $N$ such that every $18$-bilipschitz embedded cyclic subgraph in $X$ has length at most $N$. 
\end{prop}

From this we obtain a version of \Cref{thm:non-hyp_implies_star} with different constants.  Notably in Hume and Mackay's result we obtain a good multiplicative constant but a relatively large additive constant.  Whereas in our result, \Cref{thm:non-hyp_implies_star}, the size of the constants is reversed.

\begin{corollary} \label{cor.loops}
Let $K > 2$. If $G$ is a non-hyperbolic group, then for any finite generating set $S$, the Cayley graph $\Gamma(G,S)$ satisfies \eqref{condition.star} with constants $(K, 18)$. 
\end{corollary}

\begin{proof}
\Cref{prop.hm} tells us that there exists an increasing sequence of natural numbers $l_n \rightarrow \infty$ and a sequence of $18$-bilipschitz embedded cyclic subgraphs in $\Gamma(G,S)$ of length $l_n$. We may write these cyclic subgraphs as injective graph homomorphisms $\phi_n : C_{l_n} \rightarrow \Gamma(G,S)$. If $L_n = \lfloor l_n/2 \rfloor$ then any subpath in $C_{l_n}$ of length $L_n$ is a geodesic. So $C_{l_n}$ is $L_n$-locally $(1,0)$-quasigeodesic, from which it follows that its image in $\Gamma(G,S)$ is $L_n$-locally $(18,0)$-quasigeodesic. For $l_n$ sufficiently large we have $l_n \leq KL_n$, and by passing to a subsequence if necessary, we may assume that $L_n \rightarrow \infty$ is an increasing sequence. The result follows. 
\end{proof}

\subsection{Regularity}\label{sec.reg}
\begin{prop}\label{prop.condition.star}
Suppose $\Gamma(G,S)$ is a Cayley graph that satisfies \eqref{condition.star} with constants $(K,\lambda)$. Then for all $\lambda' > (2K - 1)\lambda$ the set of $(\lambda',0)$-quasigeodesics do not form a regular language.
\end{prop}

\begin{proof}
To prove the proposition we will show that any automata accepting the language of $(\lambda',0)$--quasigeodesics must have infinitely many distinct states.  In particular, the language is not accepted by an FSA and so is not regular.

Fix a generating set $S$ such that the Cayley graph $\Gamma(G,S)$ satisfies \eqref{condition.star} with constants $(K, \lambda)$. Let $\lambda' > (2K - 1)\lambda$. Write $l_n = \ell(\gamma_n)$. 

We need to choose the parametrisation of our loops $\gamma_n$ thoughtfully. 

\begin{claim*}
\emph{There exists an arclength parametrisation $\gamma_n: [0,l_n] \rightarrow \Gamma(G,S)$ of the loop $\gamma_n$ such that if $T_n \in \N$ is the minimal natural number such that $\gamma_n \eval_{[0,T_n]}$ is \textit{not} a $(\lambda',0)$-quasigeodesic, then $\gamma_n \eval_{[1,T_n]}$ is a $(\lambda',0)$-quasigeodesic.}
\end{claim*}

\begin{claimproof}[Proof of Claim]
To begin with, suppose $\gamma_n': [0,l_n] \rightarrow \Gamma(G,S)$ is some arbitrary parametrisation by arclength of the loop $\gamma_n$. Let $T_n' \in \N$ be the minimal natural number such that $\gamma_n' \eval_{[0,T_n']}$ is \textit{not} a $(\lambda',0)$-quasigeodesic. It follows that there exists a non-empty collection of non-negative integers 
\[\T_n' = \{T \in \N_0: T \leq T_n' \textrm{ and }\lambda' \absval{\gamma_n'(T)^{-1}\gamma_n'(T_n')} < T_n' - T\}\]
Let $S_n' = \max \T_n'$. Consider now the alternate parametrisation $\gamma_n: [0,l_n] \rightarrow \Gamma(G,S)$ of our loop defined by $\gamma_n(t) = \gamma_n'(t + S_n')$. It follows that $\gamma_n(0) = \gamma_n'(S_n')$ and $\gamma_n(T_n' - S_n') = \gamma_n'(T_n')$. If we define $T_n = T_n' - S_n'$ then 
\begin{itemize}
    \item $\gamma_n \eval_{[0,T_n - 1]}$ is a $(\lambda',0)$-quasigeodesic;
    \item $\gamma_n \eval_{[0,T_n]}$ is not a $(\lambda',0)$-quasigeodesic;
    \item $\gamma_n \eval_{[1,T_n]}$ is a $(\lambda',0)$-quasigeodesic;
\end{itemize}
and the claim follows. 
\end{claimproof}

Evidently, we may also assume that $\gamma_n(0) = e$ for all $n$. For each $n$, we fix the following notation:

\begin{itemize}
    \item Let $t_n$ be the maximal natural number such that $\gamma_n \eval_{[0,t_n]}$ is a $(\lambda,0)$--quasigeodesic;
    \item let $T_n$ be the minimal natural number such that $\gamma_n \eval_{[0,T_n]}$ is \textit{not} a $(\lambda',0)$--quasigeodesic;
    \item let $g_n \coloneqq \gamma_n(t_n)$;
    \item let $h_n \coloneqq \gamma_n(t_n)^{-1}\gamma_n(T_n)$. So $\gamma_n(T_n) = g_n h_n$. 
\end{itemize}

Now, let $m \in \NN$ be arbitrary and let $n$ be such that \begin{equation}L_n > \frac{2KL_m}{\kappa}.\label{eqn.defn.Ln}\end{equation}
where $\kappa$ is the positive constant
\begin{equation}
 \kappa \coloneqq \frac{1}{\lambda} - \frac{2K - 1}{\lambda'}. \label{eqn.defn.kappa}
\end{equation}

We will show that the $(\lambda',0)$--quasigeodesics $\gamma_m \eval_{[0,t_m]}$ and $\gamma_n \eval_{[0,t_n]}$ are in different states at times $t_m$ and $t_n$ respectively. 

Let $\eta: [0, t_m + T_n - t_n]$ denote the concatenation of $\gamma_m \eval_{[0,t_m]}$ with the path $\gamma_n \eval_{[t_n,T_n]}$.

Suppose first that $\eta \eval_{[0,t_m + T_n - t_n - 1]}$ is \textit{not} a $(\lambda',0)$-quasigeodesic. Then we are done since we know that $\gamma_n \eval_{[0,t_n]}$ concatenated with $\gamma_n \eval_{[t_n,T_n - 1]}$ \textit{is} a $(\lambda',0)$-quasigeodesic (by the minimality of $T_n$) whereas $\gamma_m \eval_{[0,t_m]}$ concatenated with the same path is \textit{not} a $(\lambda',0)$-quasigeodesic. So we may assume that $\eta \eval_{[0,t_m + T_n - t_n - 1]}$ is a $(\lambda',0)$-quasigeodesic. 

Suppose we have proven that $\eta \eval_{[0,t_m + T_n - t_n]}$ \textit{is} a $(\lambda',0)$-quasigeodesic. Then $\gamma_n \eval_{[0,t_n]}$ concatenated with the path $\gamma_n \eval_{[t_n, T_n]}$ \textit{is not} a $(\lambda',0)$--quasigeodesic (by the definition of $T_n$), but $\gamma_m \eval_{[0,t_m]}$ concatenated with the same path \textit{is} a $(\lambda',0)$--quasigeodesic. It follows that the $(\lambda',0)$--quasigeodesics $\gamma_m \eval_{[0,t_m]}$ and $\gamma_n \eval_{[0,t_n]}$ are in different states at times $t_m$ and $t_n$ respectively. So we would like to prove that $\eta \eval_{[0,t_m + T_n - t_n]}$ is a $(\lambda',0)$-quasigeodesic. Looking for a contradiction, suppose this is false. Then there exists some $0 \leq t \leq t_m + T_n - t_n$, $t \in \N_0$, such that 
\[\lambda' \absval{\eta(t)^{-1}\eta(t_m + T_n - t_n)} < t_m + T_n - t_n - t\]
If $t \geq t_m$, then $\gamma_n \eval_{[t,T_n]}$ is not a $(\lambda',0)$-quasigeodesic. However, by our assumption on the parametrisations of the loops, we know this is not the case. So we may assume that $t \leq t_m$. Hence, 
\begin{equation} \lambda' \absval{\gamma_m(t)^{-1} g_m h_n} < t_m + T_n - t_n - t \label{eqn.assumption}\end{equation}

To conclude the proof, we'll need to recall the following inequalities. By condition \eqref{condition.star}, we have:
\begin{align}
L_m &\leq t_m \leq KL_m;\label{eqn.fact1}\\
L_n &\leq t_n;\label{eqn.fact2}\\
T_n &\leq KL_n \leq Kt_n.\label{eqn.fact3}
\end{align}
Since the path $\gamma_m$ is simplicial, and since $\gamma_n\vert_{[0,t_n]}$ is a $(\lambda,0)$--quasigeodesic, we have:
\begin{align}
\absval{\gamma_m(t)^{-1} g_m} &\leq t_m - t;\label{eqn.fact4}\\
\frac{t_n}{\lambda} &\leq \absval{g_n}.\label{eqn.fact5}\end{align}
Finally, since $\gamma_n \eval_{[0,T_n - 1]}$ and $\gamma_n \eval_{[1,T_n]}$ are both $(\lambda',0)$-quasigeodesics yet $\gamma_n \eval_{[0,T_n]}$ is not a $(\lambda',0)$-quasigeodesic, we obtain:
\begin{align}
\absval{\gamma_n(T_n)} &< \frac{T_n}{\lambda'}.\label{eqn.fact6}
\end{align}

We have
$\absval{h_n} \geq \absval{g_n} - \absval{\gamma_n(T_n)}$,
so by \eqref{eqn.fact5} and \eqref{eqn.fact6} we see that
$\absval{h_n} \geq \frac{t_n}{\lambda} - \frac{T_n}{\lambda'}$.  It then follows from \eqref{eqn.fact3} that
\begin{equation}\absval{h_n}\geq \left(\frac{1}{\lambda} - \frac{K}{\lambda'}\right)t_n.\label{eqn.fact7}\end{equation}

Now, $\absval{\gamma_m(t)^{-1} g_m h_n} \geq \absval{h_n} - \absval{\gamma_m(t)^{-1} g_m}$, so by \eqref{eqn.fact7} and \eqref{eqn.fact4} we obtain
\begin{equation}
\absval{\gamma_m(t)^{-1} g_m h_n} \geq \left(\frac{1}{\lambda} - \frac{K}{\lambda'}\right)t_n  - (t_m - t).\label{eqn.step1}\end{equation}
Combining our assumption \eqref{eqn.assumption} with \eqref{eqn.fact3} we obtain
\begin{equation}\lambda'\absval{\gamma_m(t)^{-1} g_m h_n}\leq t_m - t + (K-1)t_n.\label{eqn.step2}\end{equation}

Next, combining \eqref{eqn.step1} and \eqref{eqn.step2} yields
\begin{align}
\frac{t_m - t + (K-1)t_n}{\lambda'} &\geq \left(\frac{1}{\lambda} - \frac{K}{\lambda'}\right)t_n  - (t_m - t).\nonumber
\intertext{This rearranges to}
0 &\geq \left(\frac{1}{\lambda} - \frac{(2K - 1)}{\lambda'}\right)t_n - \left(1 + \frac{1}{\lambda'}\right)(t_m-t);\nonumber\\
&\geq \kappa t_n - 2(t_m-t),\nonumber
\intertext{where $\kappa$ is defined in \eqref{eqn.defn.kappa}.  Now,}
 t_n &\leq \frac{2(t_m - t)}{\kappa} \leq \frac{2t_m}{\kappa},\nonumber\\
\intertext{and so by \eqref{eqn.fact1} and \eqref{eqn.fact2} we have}
 L_n &\leq \frac{2KL_m}{\kappa}\nonumber
\end{align}
which contradicts \eqref{eqn.defn.Ln}. So $\eta \eval_{[0,t_m + T_n - t_n]}$ is a $(\lambda',0)$-quasigeodesic and so the $(\lambda',0)$--quasigeodesics $\gamma_m \eval_{[0,t_m]}$ and $\gamma_n \eval_{[0,t_n]}$ are in different states at times $t_m$ and $t_n$ respectively.

Let $\xi: \NN \rightarrow \NN$ be the function 
\[ \xi(m) = \min \left\{ n \in \NN : L_n > \frac{2KL_m}{\kappa}\right\}.\]
Let $(n_i)_{i \in \N}$ be the integer sequence defined inductively by $n_1 = 1$, $n_{i+1} = \xi(n_i)$. By the above, we know that for every $i \in \N$ and for every $j < i$, $\gamma_{n_i}(t_{n_i})$ is in a different state to $\gamma_{n_j}(t_{n_j})$ as a $(\lambda',0)$-quasigeodesic. It follows that there are infinitely many different $(\lambda',0)$-states. Hence, the $(\lambda',0)$--quasigeodesics in $\Gamma(G,S)$ cannot form a regular language.
\end{proof}

\begin{duplicate}[\Cref{thm.converse}]
Let $G$ be a finitely generated non-hyperbolic group. Then for all finite generating sets $S$ and all $\lambda' > 54$ the set of $(\lambda',0)$-quasigeodesics is not regular in $\Gamma(G,S)$. 
\end{duplicate}
\begin{proof}
Since $G$ is not hyperbolic, by \Cref{thm:non-hyp_implies_star}, for each finitely generated Cayley graph $\Gamma$ of $G$, there exists some $K>1$ such that $\Gamma$ satisfies \eqref{condition.star} with constants $(K,3)$.  Instead, by \Cref{cor.loops}, we may take the constants $(K,18)$ with $K>2$ --- we choose to work with these constants.  Now, \Cref{prop.condition.star} implies that for all $\lambda'>(2\times 2-1)\times 18=54$, the set of $(\lambda',0)$-quasigeodesics is not regular in $\Gamma(G,S)$, as required.
\end{proof}

\begin{duplicate}[\Cref{thm.main}]
    A finitely generated group is hyperbolic if and only if it is $\QREG$.
\end{duplicate}
\begin{proof}
That hyperbolicity implies $\QREG$ was proven by Holt and Rees in \cite{HoltRees2003}. The other direction is given by our \Cref{thm.converse}.
\end{proof}

\Cref{thm.converse} suggests that for every non-hyperbolic group $G$ and every finite generating set $S$ there exists some infimal value of $\lambda$ such that for all $\lambda' > \lambda$ the language of $(\lambda',0)$-quasigeodesics in $\Gamma(G,S)$ is not regular. Given a non-hyperbolic group $G$, it is interesting to ask what this infimal $\lambda$ might be. As far as we are aware, there are no known examples of non-hyperbolic Cayley graphs $\Gamma(G,S)$ with regular $(\lambda,0)$-quasigeodesics unless $\lambda = 1$. 

\begin{conjecture}
If $G$ is a non-hyperbolic finitely generated group, then for all generating sets $S$, and for all $\lambda' > 1$, the $(\lambda',0)$-quasigeodesics in $\Gamma(G,S)$ do not form a regular language. 
\end{conjecture}

\subsection{Geodetic Cayley graphs}\label{sec.geod}

Finally, we prove our application to Shaprio's question on geodetic Cayley graphs.

\begin{duplicate}[\Cref{thm:UG_Hyperbolic}]
Let $G$ be a finitely generated group with a generating set $S$ such that $\Gamma(G, S)$ is geodetic.  If there exists $\lambda>3$ such that the language of $(\lambda, 0)$--quasigeodesics is regular, then $G$ is hyperbolic and hence virtually free.
\end{duplicate}

\begin{proof}
Let $\Gamma$ be a graph. An \emph{isometrically embedded circuit} (IEC) is a simplicial loop of length $2n+1$ such that the restriction of each subsegment of length at most $n$ is a geodesic. 

We claim that if $\Gamma$ is geodetic and not hyperbolic then there are IECs of arbitrarily large length. So suppose $\Gamma$ is geodetic and there is a uniform bound on the length of IECs. We will show that $\Gamma$ is hyperbolic. By \cite[Theorem~1.4]{Papasoglu1995b}, a Cayley graph is hyperbolic if and only if all geodesic bigons are uniformly thin, i.e. any two geodesics sharing endpoints have uniformly bounded Hausdoff distance. Consider an arbitrary geodesic bigon in $\Gamma$. Since $\Gamma$ is geodetic, if the geodesic endpoints are vertices, the geodesics need to coincide. Further, it is straightforward to check that if the endpoints are both in an edge one can reduce to a case where at least one endpoint is a vertex. So, the only case left is a bigon where one endpoint is a vertex and the other is the midpoint of an edge. By \cite[Lemma~4]{ElderPiggott2020}, such a configuration produces an IEC. Since these have uniformly bounded length, the bigon is thin. So $\Gamma$ is hyperbolic and the claim is proved.

Observe that an IEC of length $2n+1$ is an $n$--local geodesic. By Proposition~\ref{prop.condition.star}, we conclude that if a group is non-hyperbolic and geodetic, then for any $\lambda'>3$ the set of $(\lambda', 0)$--quasigeodesics cannot form a regular language. \end{proof}

\bibliographystyle{halpha}
\bibliography{main}

\end{document}

%% file: top.tex
\usepackage{amsmath} 
\usepackage{amssymb} 
\usepackage{amsthm} 
\usepackage[english]{babel} 
\usepackage[font=small,justification=centering]{caption} 
\usepackage[nodayofweek]{datetime}
\usepackage[shortlabels]{enumitem} 
\usepackage[T1]{fontenc} 
\usepackage[utf8]{inputenc} 
\usepackage{ifthen} 
\usepackage{mathabx} 
\usepackage{mathtools} 
\usepackage[dvipsnames]{xcolor} 
\usepackage[pdftex,  colorlinks=true, pagebackref=true]{hyperref} 
    \hypersetup{urlcolor=RoyalBlue, linkcolor=RoyalBlue,  citecolor=black}
\usepackage{setspace} 
\usepackage{tikz-cd} 
\usepackage{xfrac} 
\usepackage{cleveref}
\usepackage{crossreftools} 
\usepackage{stmaryrd} 
\usepackage{comment} 

\renewcommand*{\backref}[1]{}
\renewcommand*{\backrefalt}[4]
{
    \ifcase #1
        No citation in the text.
    \or
        Cited on Page #2.
    \else
        Cited on Pages #2.
    \fi
}

\makeatletter
\def\@tocline#1#2#3#4#5#6#7{\relax
  \ifnum #1>\c@tocdepth 
  \else
    \par \addpenalty\@secpenalty\addvspace{#2}%
    \begingroup \hyphenpenalty\@M
    \@ifempty{#4}{%
      \@tempdima\csname r@tocindent\number#1\endcsname\relax
    }{%
      \@tempdima#4\relax
    }%
    \parindent\z@ \leftskip#3\relax \advance\leftskip\@tempdima\relax
    \rightskip\@pnumwidth plus4em \parfillskip-\@pnumwidth
    #5\leavevmode\hskip-\@tempdima
      \ifcase #1
       \or\or \hskip 1em \or \hskip 2em \else \hskip 3em \fi%
      #6\nobreak\relax
    \dotfill\hbox to\@pnumwidth{\@tocpagenum{#7}}\par
    \nobreak
    \endgroup
  \fi}
\makeatother

\newtheorem{thm}{Theorem}[section]

\newtheorem{corollary}[thm]{Corollary}
\newtheorem{prop}[thm]{Proposition}
\newtheorem{conjecture}[thm]{Conjecture}

    \newtheoremstyle{TheoremNum}
        {8.0pt plus 2.0pt minus 4.0pt}{8.0pt plus 2.0pt minus 4.0pt} 
        {\itshape} 
        {-0.15cm} 
        {\bfseries} 
        {.} 
        { }  
        {\thmname{#1}\thmnote{ \bfseries #3}}
    \theoremstyle{TheoremNum}
    \newtheorem{duplicate}{}

\theoremstyle{definition}
\newtheorem{defn}[thm]{Definition}

\newcommand*{\claimproofname}{My proof}
\newenvironment{claimproof}[1][\claimproofname]{\begin{proof}[#1]}{\end{proof}}

\DeclareMathOperator{\diam}{diam}

\newcommand{\calm}{{\mathcal{M}}}

\newcommand{\calx}{{\mathcal{X}}}

\newcommand{\NN}{\mathbb{N}}

\newcommand{\RR}{\mathbb{R}}
\newcommand{\QQ}{\mathbb{Q}}

\usepackage{tikz}
\usetikzlibrary{arrows,quotes}
\tikzstyle{blackNode}=[fill=black, draw=black, shape=circle]